\documentclass[11 pt]{article}
\usepackage{graphicx}
\usepackage{amsmath,latexsym,amssymb,amsthm,enumerate,amscd}

\theoremstyle{plain}
\newtheorem{theorem}{Theorem}[section]
\newtheorem{proposition}[theorem]{Proposition}

\newtheorem{lemma}[theorem]{Lemma}
\theoremstyle{definition}
\newtheorem{definition}[theorem]{Definition}

\newtheorem{remark}[theorem]{Remark}
\newtheorem{example}[theorem]{Example}

\theoremstyle{plain}

\newtheorem{notation}[theorem]{Notation}

\numberwithin{equation}{section}
\numberwithin{table}{section} 
\def\<{\left<}
\def\>{\right>}
\def\k{\mathsf{k}}

\title{The poset of the nilpotent commutator of a nilpotent matrix}
\author{Leila Khatami} 

\begin{document}
\maketitle
\let\thefootnote\relax
\let\thefootnote\relax

\begin{abstract}
Let $B$ be an $n \times n$ nilpotent matrix with entries in an infinite field $\k$. Assume that $B$ is in Jordan canonical form with the associated Jordan block partition $P$. In this paper, we study a poset $\mathcal{D}_P$ associated to the nilpotent commutator of $B$ and a certain partition of $n$, denoted by $\lambda_U(P)$,  defined in terms of the lengths of unions of special chains in $\mathcal{D}_P$. Polona Oblak associated to a given partition $P$ another partition $\mbox{Ob}(P)$ resulting from a recursive process. She conjectured that $\mbox{Ob}(P)$ is the same as the Jordan partition $Q(P)$ of a generic element of the nilpotent commutator of $B$. Roberta Basili, Anthony Iarrobino and the author later generalized the process introduced by Oblak. In this paper we show that all such processes result in the partition $\lambda_U(P)$. 

%

\end{abstract}
\section*{Introduction}
Let $\k$ be an infinite field and  $B$ a nilpotent $n\times n$ matrix with entries in $\k$. Suppose that $B$ is in  Jordan canonical form with associated Jordan block partition $P$. Recall that the centralizer and the nilpotent centralizer of $B$ are, respectively, defined as follows:

$$\begin{array}{ll}
\mathcal{C}_B=&\{A \in \mathcal{M}at_n(\k) \, \mid  \, AB=BA \}, \\
\mathcal{N}_B=&\{A \in \mathcal{M}at_n(\k) \, \mid  \, AB=BA \mbox{ and } A \mbox{ is nilpotent} \}.
\end{array}$$
Here $\mathcal{M}at_n(\k)$ denotes the set of all $n \times n$ matrices with entires in $\k$.

It is well known that $\mathcal{N}_B$ is an irreducible algebraic variety (see \cite[Lemma 1.5]{BI}). Therefore, there is a unique partition of $n$ corresponding to the Jordan type of a generic element of $\mathcal{N}_B$. We denote this unique partition by $Q(P)$. The map $P \to Q(P)$ has been studied by different authors (see \cite{BI}, \cite{BIK}, \cite{KO}, \cite{Oblak}, and \cite{Pa}). It is known, by the work of T.~ Ko\v{s}ir and P. Oblak (\cite{KO}), using also a result of R. Basili and A. Iarrobino (\cite{BI}), that if $\k$ has characteristic zero then the map $P \to Q(P)$ is idempotent: $Q(Q(P))=Q(P)$. The number of parts of the partition $Q(P)$ is also completely determined by R. Basili (\cite[Proposition 2.4]{Basili03} and \cite[Theorem 2.17]{BIK}). In \cite[Theorem 6]{Oblak}, P. Oblak obtains a formula for the index-- largest part -- of the partition $Q(P)$ when char $\k=0$. Her result is generalized over any infinite fleld $\k$ in \cite{IK} by A. Iarrobino and the author.

In this paper, we work with a poset $\mathcal{D}_P$ determined by the partition $P$. The poset is closely connected to $\mathcal{U}_B$, a maximal nilpotent subalgebra of the centralizer $\mathcal{C}_B$. The poset $\mathcal{D}_P$ and the subalgebra $\mathcal{U}_B$ were implicitly used in \cite{KO} and \cite{Oblak}, and were defined in \cite{BIK}.

We review the definition of $\mathcal{D}_P$ in the first section and also recall the classical partition invariant $\lambda(P)=\lambda(\mathcal{D}_P)$ of the poset $\mathcal{D}_P$, defined in terms of the lengths of unions of chains in $\mathcal{D}_P$. We then define and study a partition, $\lambda_U(P)$,  associated to the poset $\mathcal{D}_P$ and always dominated by $\lambda(P)$. This new partition is also defined in terms of the lengths of unions of chains in $\mathcal{D}_P$, but this time the choice of chains is restricted to special types of chains that we call $U$-chains. The $U$-chains are closely related to a recursive process introduced by P.Oblak and generalized in \cite{BIK}. In Theorem \ref{unique Oblak} we prove that any such process gives rise to the partition $\lambda_U(P)$. P. Oblak also conjectured that the partition resulting from the process she suggested is the same as $Q(P)$ and in \cite{Oblak} she proves her conjecture for a partition $P$ such that $Q(P)$ has at most 2 parts and $\k=\mathbb{R}$. In \cite{IK}, we show for an infinite field $\k$  that $\lambda_U(P)$ is always dominated by $Q(P)$, which proves ``half" of Oblak's conjecture. In \cite{minpart}, we will give an explicit formula for the smallest part of $\lambda_U(P)$ and prove that it is the same as the smallest part of $\lambda(P)$, and thus also $Q(P)$, by results of \cite{IK}. Thus we give an explicit formula for $Q(P)$ when it has at most 3 parts. 


\medskip

\noindent{\bf Acknowledgement.} The author is grateful to A. Iarrobino for invaluable discussions on the topic, as well as for his comments and suggestions on the paper. The author is also thankful to Bart Van Steirteghem and to Toma\v{z} Ko\v{s}ir for their thorough comments on a draft of this paper.

\section{Poset $\mathcal{D}_P$ and $U$-chains}\label{D}

\noindent {\bf Notation.} Throughout this paper $n$ will denote a positive integer and $P$ a partition of $n$. For any positive integer $p$, $n_p \geq 0$ denotes the multiplicity of the part $p$ in $P$.  
\par
Let $V$ be an $n$-dimensional $\k$-vector space and fix a nilpotent linear transformation $T \in \mbox{End}_{\k}(V)$. Let $B$ be the Jordan canonical form of $T$ with Jordan block partition $P=(p_s^{n_{p_s}}, \cdots, p_1^{n_{p_1}})$ such that $p_s>\cdots>p_1$. So there is a decomposition of $V$  into $B$-invariant subspaces,  $$V=\oplus V_{p_i,k}, \, \, \, \, 1 \leq k \leq n_{p_i} \mbox{ and }1\leq i \leq s.$$ For each $p_i$ and each $1 \leq k \leq n_{p_i}$, we choose a cyclic vector $(1,p_i,k)$ for $V_{p_i,k}$, which determines the basis 
\begin{equation}\label{basis}
\{(u,p_i,k)=B^{u-1}(1,p_i,k) \, \mid  \, u=1, \cdots, p_i\}
\end{equation} 
for $V_{p_i,k}$.  Let $W_{i}$ be the subspace of $V$ spanned by the cyclic vectors $(1,p_i,k)$, where $k \in \{1, \cdots, n_{p_i}\}$.

Define $\pi_i: \mathcal{C}_B \to \mbox{End}_{\k}(W_i)\cong\mathcal{M}at_{n_{p_i}}(\k)$ by sending a matrix $C \in \mathcal{C}_B$ to the endomorphism obtained by first restricting $C$ to $W_i$ and then projecting to $W_i$. It is well known that, up to isomorphism, the map $$\pi=\prod_{i=1}^s\pi_i :\mathcal{C}_B \to \prod_{i=1}^s  \mbox{End}_{\k}(W_i)$$ is the canonical projection from $\mathcal{C}_B$ to its semi-simple quotient (see \cite[Lemma 2.3]{Basili03}, \cite[Theorem 2.3]{BIK}, \cite[Theorem 6]{HW}). 

\begin{definition}
For each $1\leq i \leq t$, let  $\frak{U}_i \subset  \mbox{End}_{\k}(W_i)$ denote the set of all strictly upper triangular elements of $ \mbox{End}_{\k}(W_i)$. Set $\frak{U} =\prod_{i=1}^t  \frak{U}_i$ and $\mathcal{U}_B=\pi^{-1}(\frak{U})$. 
\end{definition}

It is easy to see that for any element $N \in \mathcal{N}_B$, there is a unit $C \in \mathcal{C}_B$ such that $CNC^{-1} \in \mathcal{U}_B$ (see \cite[Lemma 2.2]{BIK}). Thus the Jordan partition of a generic element of $\mathcal{N}_B$ is that of a generic element of $\mathcal{U}_B$.

To a partition $P$, we associate a poset $\mathcal{D}_P$ whose elements are the basis for $V$ from equation \ref{basis}. We next define the partial order on $\mathcal{D}_P$, which will satisfy, for all $v, v' \in \mathcal{D}_P$
\begin{equation}\label{U and D}
v \leq v' \iff \exists A \in \mathcal{U}_B \mbox{ such that } A \, v \, \mid _{v'} \neq 0 \, \,\,\,(\mbox{See }\cite[\mbox{Equation } 2.18]{BIK}).
\end{equation}

We visualize $\mathcal{D}_P$ by its covering edge diagram, a digraph, which we will also denote by $\mathcal{D}_P$. We say that the vertex $v'$ {\it covers}  the vertex $v$ if $ v<v'$ and there is no vertex $v''$ with $v<v''< v'$. There is an edge from $v$ to $v'$ in the digraph if and only if $v'$ covers $v$.

\begin{definition}\label{poset}

Let $P=(p_s^{n_{p_s}}, \cdots, p_1^{n_{p_1}})$ be a partition of $n$ with $p_s>\cdots>p_1$  and $n_{p_i} > 0$ for $1 \leq i \leq s.$ We define the [covering edge] diagram of $\mathcal{D}_P$ as follows. (See Figure \ref{poset figure}.)

\begin{itemize}
\item Vertices of the diagram of $\mathcal{D}_P$:

 For each $1 \leq i \leq t$, there are $n_{p_i}$ rows each consisting of $p_i$ vertices labeled by triples $(u, p_i, k)$ such that $1\leq u \leq p_i$ and $1 \leq k \leq n_{p_i}$. For each $p_i$, we arrange the vertices in a way that the first and last components of the triple are increasing when we go from left to right and from bottom to top, respectively. 
 
 We say that a vertex of the form $(u,p_i,k)$ is a vertex in {\it level} $p_i$. 
 
\item Covering edges of the diagram of $\mathcal{D}_P$: 
\begin{itemize}
\item[i.] For $1<i \leq s$, the edge $\beta_{p_{i}, p_{i-1}}$ from the top vertex $(u,p_i,n_{p_i})$ of any column in the rows corresponding to $p_i$ to the bottom vertex $(u, p_{i-1}, 1)$ in the rows corresponding to $p_{i-1}$. 

\item[ii.] For $1\leq i< s$, the edge $\alpha_{p_{i}, p_{i+1}}$ from the top vertex $(u,p_i,n_{p_i})$ of any column in the rows corresponding to $p_i$ to the bottom vertex $(u+p_{i+1}- p_{i}, p_{i+1}, 1)$ in the rows corresponding to $p_{i+1}$.

\item[iii.] For $1 \leq i \leq s$, $1 \leq u \leq p_i$ and $1 \leq k < n_{p_i}$, the upward arrow $e_{(u,p_i,k)}$ from $(u,p_i, k)$ to $(u,p_i, k+1)$ .

\item[iv.] For any isolated $p_i$ ({\it i.e.} $p_{i+1}-p_i>1$ and $p_i-p_{i-1}>1$) and any $1 \leq u <p_i$, $\omega_{p_i}$ from $(u, p_i, n_{p_i})$ to $(u+1, p_i, 1)$.

\end{itemize}
\end{itemize}

\end{definition}

\begin{figure}[htb]\label{poset figure}
\begin{center}
\includegraphics[scale=.3]{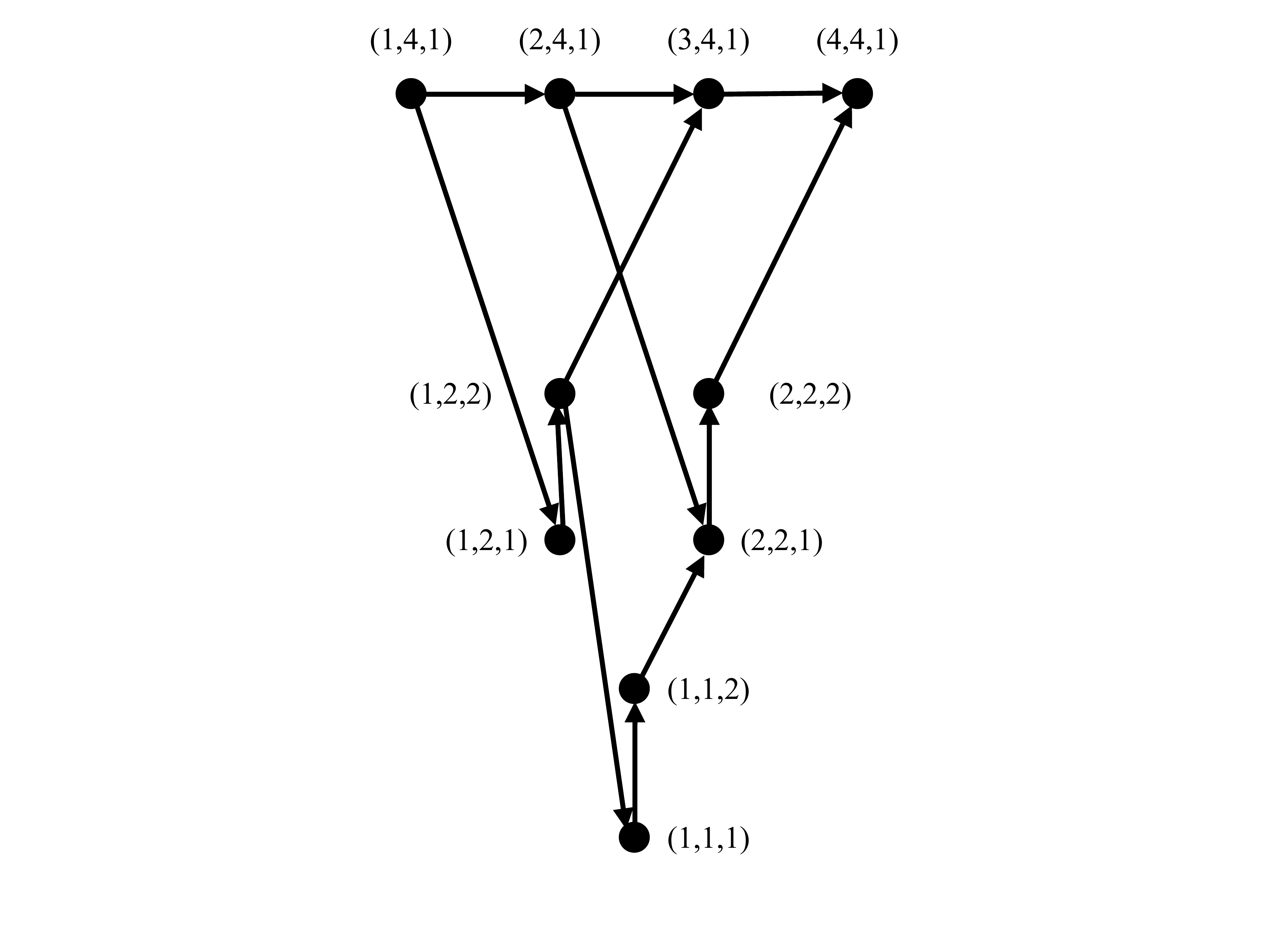}
\end{center}
\vspace{-.2in}
\caption{Poset $\mathcal{D}_P$ for $ P=(4,2^2,1^2)$}
\end{figure}

We will associate to $\mathcal{D}_P$ several partitions. The first is the classical partition associated to a poset, introduced by Greene, Saks and others, and used in different areas of discrete mathematics and algebraic geometry(see \cite{Greene}, \cite{Saks} and the excellent exposition in \cite{Britz-Fomin}). 

\par 

Recall that a {\it chain} is a totally ordered subset of a poset $\mathcal{D}$, whose {\it length} is its cardinality. We say a chain $C$ is {\it maximum}, if it has the maximum cardinality among all chains of the poset.

\begin{definition}\label{lambda}
To a poset $\mathcal{D}$ of cardinality $n$, the partition  $\lambda(\mathcal{D})$ of $n$ is assigned as follows. For $k=0, 1, \cdots$, let $c_k$ denote the maximum cardinality of a union of $k$ chains in $\mathcal{D}$. Let $\lambda_k=c_k-c_{k-1}$ for all $k \geq 1$ and define $\lambda(\mathcal{D})=(\lambda_1, \lambda_2, \cdots)$.
\end{definition}

\begin{notation} Suppose that $P$ is a partition of $n$ and $\mathcal{D}_P$ is the corresponding poset. We denote $\lambda(\mathcal{D}_P)$, by $\lambda(P)$. 
\end{notation}

\begin{definition}
A partition is {\it almost rectangular} if its biggest and smallest parts differ by at most one. 

Note that any partition $P$ can be written as $P(1) \cup \cdots \cup P(r)$, where each $P(i)$ is an almost rectangular subpartition. The minimum number $r$ in any such decomposition is denoted by $r_P$. 
\end{definition} 
In \cite[Proposition 2.4]{Basili03} and \cite[Theorem 2.17]{BIK}, it is proved that $Q(P)$ has exactly $r_P$ parts.

\begin{example}
Partition $P=(3,3,2,2,2)$ is almost rectangular and in particular $r_P=1$. As for $Q=(7, 2,2,1)$, we have $r_Q=2$.
\end{example}

\begin{definition}\label{general U chain}
Let $P=(\ldots, p^{n_p}, \ldots)$ be a partition of $n$ (here $n_p\geq 0)$. For a positive integer $r$ and a set $\frak{A}=\{a_1,a_1+1, \cdots , a_r,a_r+1\}\subset\mathbb{N}$ such that $a_1<a_1+1 <\cdots< a_r< a_{r}+1$, we define the {\it $r$-$U$-chain} $U_{\frak{A}}$ as follows: 
$$\begin{array}{c}
U_{\frak{A}}=\cup_{i=1}^{r}S_{{\frak{A}};i},\mbox{ where}\\
S_{{\frak{A}};i}=\{(u,p,k)\in \mathcal{D}_P\,\mid \, p\in\{a_i, a_i+1\} \mbox{ and } i\leq u \leq p-i+1\}\\
\cup \{(u,p,k)\in \mathcal{D}_P\,\mid \, p>a_i+1 \mbox{ and } u\in\{i,p-i+1\} \}.
\end{array}$$
Note that each $S_{{\frak{A}},i}$ is a chain in $\mathcal{D}_P$ and that $S_{{\frak{A}},i}\cap S_{{\frak{A}},j}=\emptyset$ if $i\neq j$. A $1$-$U$-chain is called a {\it simple $U$-chain}.
\end{definition}

\begin{notation}
If $\frak{A}=\{a_1,a_1+1, \cdots , a_r,a_r+1\}$, then we often denote $U_\frak{A}$ by $U_{a_1, \cdots , a_r}$. 
\end{notation}

%

\begin{figure}\label{shelling}
\begin{tabular}{cc}
\includegraphics[scale=.2]{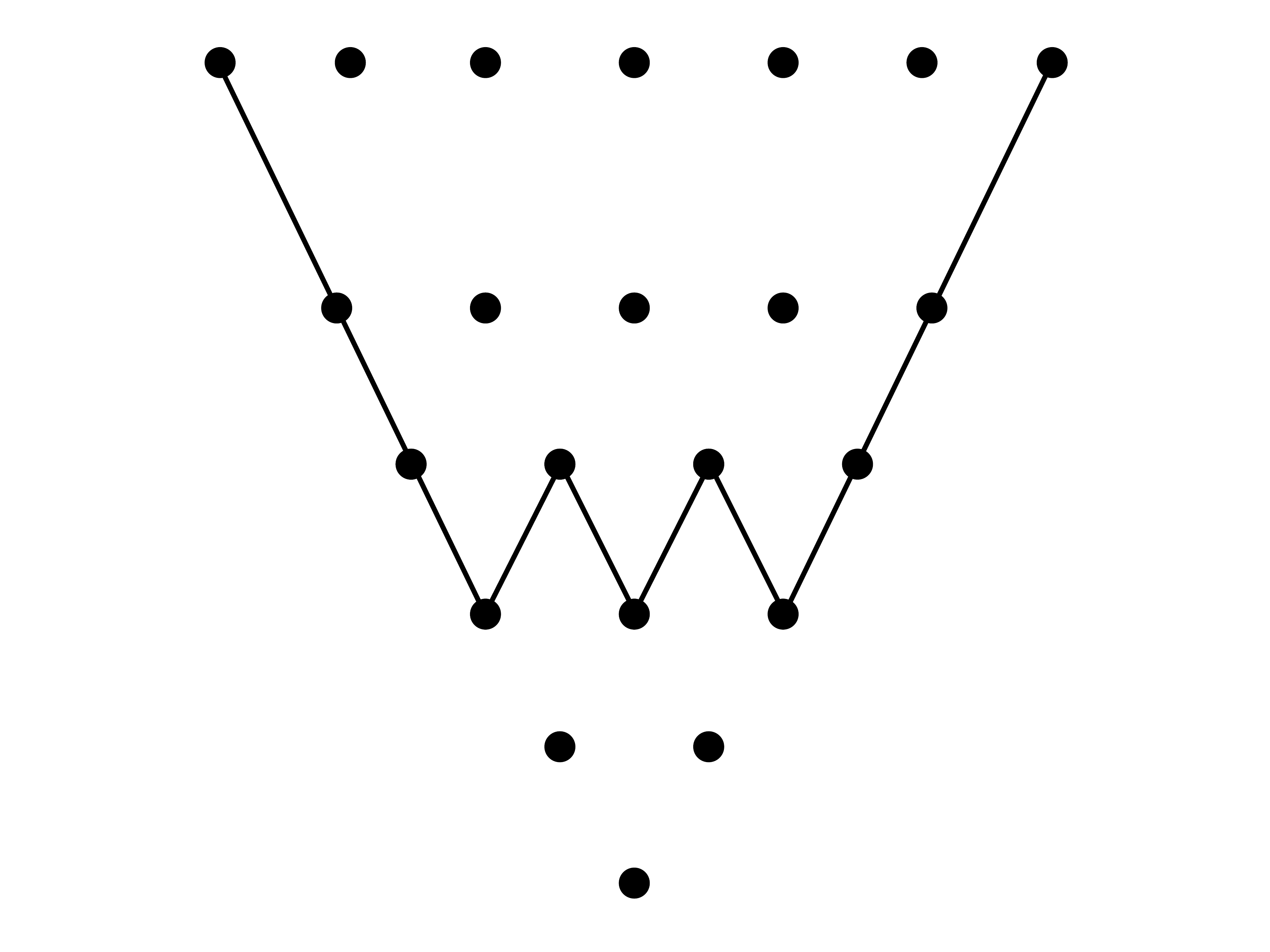}&\includegraphics[scale=.2]{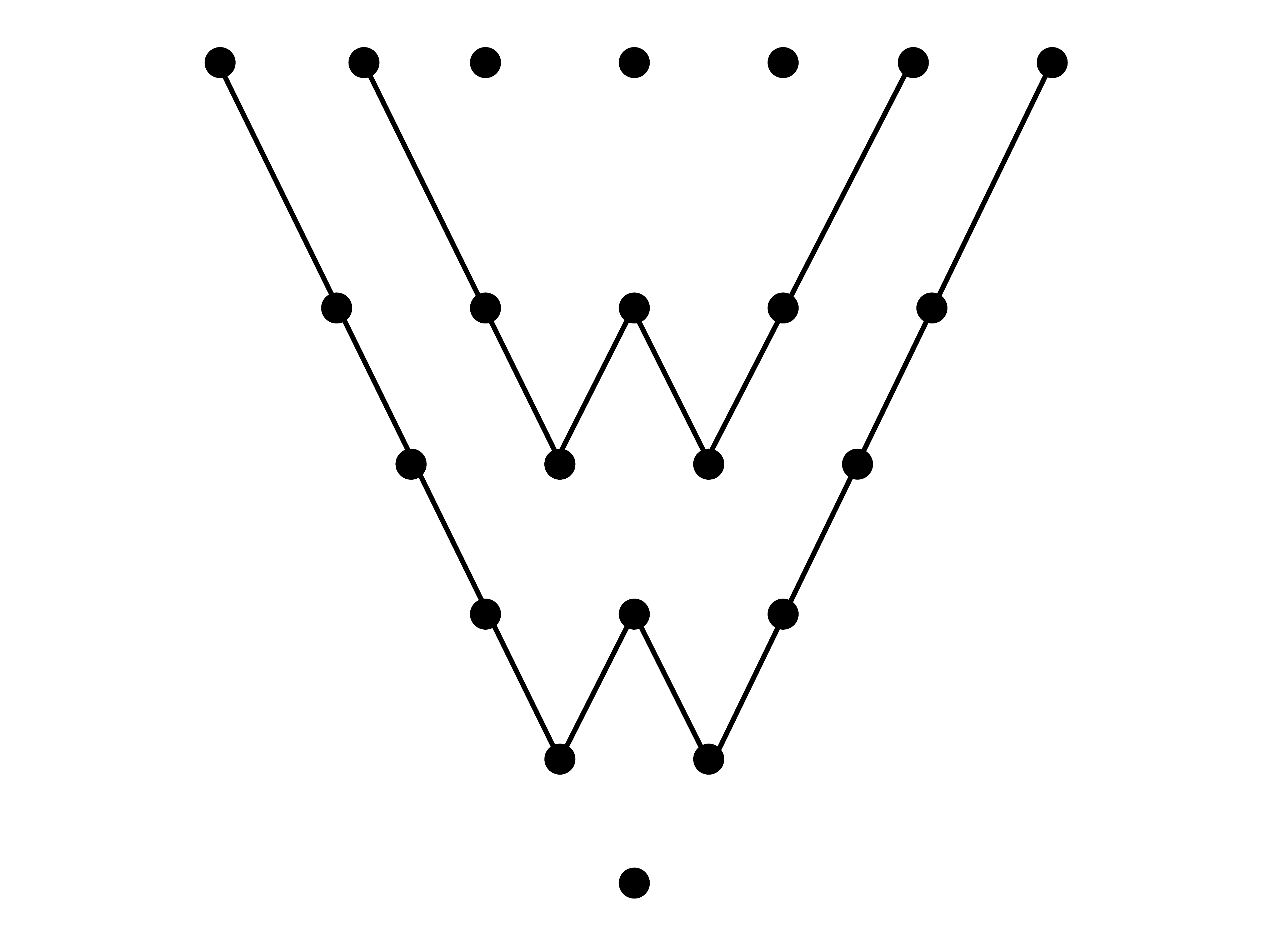}
\end{tabular}
\vspace{-.3in}
\caption{$U$-Chains $U_3$ and $U_{2,4}$ in $\mathcal{D}_P$ with $ P=(7,5,4,3,2,1)$}
\end{figure}

\begin{example}
In Definition \ref{general U chain} above, $\frak{A}$ need not be a subset of $\{p_1, \cdots, p_s\}$. For example, let $P=(7,5,4,3,2,1)$, then $U_6$ is a simple $U$-chain in $\mathcal{D}_P$. We have 
$$U_6=U_{\{6,7\}}=\{(u,7,1)\,\mid \, 1\leq u \leq 7\}.$$

Figure 2 illustrates the simple $U$-chain $U_3=U_{\{3,4\}}$, and the 2-$U$-chain $U_{2,4}=U_{\{2,3,4,5\}}$ in $\mathcal{D_P}$.
%
%
\end{example}

\begin{definition}\label{lambda U}
Let $P$ be a partition of $n$. Define a partition $\lambda_U(P)=(\lambda_{U,1}, \lambda_{U,2}, \cdots)$ of $n$, such that $\lambda_{U,k}=u_{k}-u_{k-1},$ where $u_k$ is the maximum cardinality of a $k$-$U$-chain in $\mathcal{D}_P$.
\end{definition}

We recall the dominance partial order on the set of all partitions of $n$.
\begin{definition}
Let $P=(p_1, p_2, \cdots)$ and  $Q=(q_1, q_2, \cdots)$, with $p_1 \geq  p_2\geq \cdots $ and $ q_1 \geq  q_2\geq \cdots $, be two partitions of $n$. Then $P \leq Q$ if and only if for $k=1, 2, \cdots$, $$\sum_{i=1}^{k}p_i \leq \sum_{i=1}^{k}q_i.$$
\end{definition}

\begin{lemma}\label{lmbdaU less lambda}
For any partition $P$ we have $\lambda_U(P) \leq \lambda(P).$
\end{lemma}
\begin{proof}
By Definition \ref{general U chain} a $k$-$U$-chain $U_\frak{A}$ is the union of $k$ chains $S_{\frak{A},i}$ for $i=1, \cdots, k$. So we always have $c_k \geq u_k$ .
\end{proof}

The following is a preparatory lemma for showing Proposition \ref{U chain}, a key replacement result.

\begin{lemma}\label{chain length lem}
Let $n>1$ and $P=(\ldots, p^{n_p}, \ldots)$ be a partition of $n$. Let  $\frak{A}=\{a_1,a_1+1, \cdots, a_r,a_r+1\}$ and $U_\frak{A}$ be the associated $r$-$U$-chain in $\mathcal{D}_P.$  For each $i\in\{1, \cdots, r\}$, let $\frak{A}_{\hat{i}}=\frak{A}\setminus\{a_i,a_i+1\}$. Then $$\mid U_\frak{A}\mid =\mid U_{\frak{A}_{\hat{i}}}\mid +\mid U_{a_i}\mid -2(i-1)(n_{a_i}+n_{a_i+1})-2\sum_{j=i+1}^r(n_{a_j}+n_{a_j+1}).$$
%
\end{lemma}

\begin{proof}

By Definition \ref{general U chain}, $U_{\frak{A}_{\hat{i}}}\subseteq U_\frak{A}$ and 
$$\begin{array}{ll}
U_\frak{A} \setminus U_{\frak{A}_{\hat{i}}}&=\{(u,p,k)\in \mathcal{D}_P\,\mid \, p\in\{a_i, a_i+1\} \mbox{ and } i\leq u \leq p-i+1\}\\
&\cup_{j=i}^r \{(u,p,k)\in \mathcal{D}_P\,\mid \, a_j+1<p<a_{j+1} \mbox{ and } u\in\{j,p-j+1\} \}.
\end{array}$$

%
%
%
%
%

Therefore 
$$\begin{array}{ll}
\mid U_\frak{A}\mid -\mid U_{\frak{A}_{\hat{i}}}\mid &= (a_i-2i+2)n_{a_i}+(a_i-2i+3)n_{a_i+1}\\
&+2\displaystyle{\sum_{p>a_i+1}n_p\, \, -2\sum_{j=i+1}^r(n_{a_j}+n_{a_j+1})}.
\end{array}$$

To complete the proof, it is enough to use Definition \ref{general U chain} again to get $$\mid U_{a_i}\mid =\mid U_{\{a_i,a_i+1\}}\mid =a_i\,n_{a_i}+(a_i+1)n_{a_i+1}+2\displaystyle{\sum_{p>a_i+1}n_p}.$$
%
\end{proof}
%
%

\begin{proposition}\label{U chain}
Let $P=(\ldots, p^{n_p}, \ldots)$ be a partition of $n>1$ and suppose that $U_a$ is a maximum simple $U$-chain in $\mathcal{D}_P$. If $U_{b_1, \cdots, b_r}$ is an $r$-$U$-chain in $\mathcal{D}_P$, then there exists $1\leq u \leq r$ such that $b_{u-1}<a<b_{u+1}-1$ and $\mid U_{b_1, \cdots, b_r}\mid  \leq  \mid U_{b_1,\cdots, b_{u-1}, a, b_{u+1}, \cdots, b_r}\mid .$
\end{proposition}
In other words, in $U_{b_1, \cdots b_r}$ we can replace {\it some} $b_u$ by $a$ and get a $U$-chain which has at least the same cardinality.

\begin{proof}
First note that by Definition \ref{general U chain}, there is nothing to prove if $\{a, a+1\} \subset \{b_1, b_1+1, \cdots, b_r, b_r+1\}$. So we assume that $\{a, a+1\} \not \subset \{b_1, b_1+1, \cdots, b_r, b_r+1\}$. Also note that since $b_1>0$, for any $u\in\{1, \cdots, r\},$ \begin{equation}\label{bu}b_u>2(u-1)\end{equation}

\noindent{\bf Case 1.} If $a\leq b_1$, then by Lemma \ref{chain length lem}, 
$$\begin{array}{ll}
\mid U_{b_1, \cdots, b_r}\mid &=\mid U_{b_2, \cdots, b_r}\mid +\mid U_{b_1}\mid -2\displaystyle{\sum_{i=2}^r(n_{b_i}+n_{b_i+1})}, \mbox{ and}\\
\mid U_{a, b_2, \cdots, b_r}\mid &=\mid U_{b_2, \cdots, b_r}\mid +\mid U_{a}\mid -2\displaystyle{\sum_{i=2}^r(n_{b_i}+n_{b_i+1})}.
\end{array}$$
Therefore $\mid U_{a, b_2, \cdots, b_r}\mid -\mid U_{b_1, \cdots, b_r}\mid =\mid U_a\mid -\mid U_{b_1}\mid  \geq 0,$ by the maximality of $\mid U_a\mid $.
\\

\noindent{\bf Case 2.} If $b_u < a < b_{u+1}$ for some $u\in \{1, \cdots, r\}$. (We set $b_{r+1}=\infty$.)

{\bf Case 2.1.} If $b_{u+1}=a+1$. 

Then $b_u<a<b_{u+1}<b_{u+2}-1$.  Since $U_a$ is a maximum simple $U$-chain, we also have
$$0\leq \mid U_a\mid -\mid U_{b_{u+1}}\mid =\mid U_a\mid -\mid U_{a+1}\mid =a\,(n_{a}-n_{a+2}).$$ 
On the other hand, by Lemma \ref{chain length lem},
$$\begin{array}{rl}
\mid U_{b_1, \cdots, b_r}\mid &=\mid U_{b_1, \cdots, b_{u},b_{u+2}, \cdots, b_r}\mid +\mid U_{b_{u+1}}\mid \\&-2u\,(n_{b_{u+1}}+n_{b_{u+1}+1})-2\displaystyle{\sum_{i=u+2}^r(n_{b_i}+n_{b_i+1})}, \mbox{ and}\\
\mid U_{b_1, \cdots,b_{u},a,b_{u+2},\cdots,  b_r}\mid &=\mid U_{b_1, \cdots, b_{u},b_{u+2}, \cdots, b_r}\mid +\mid U_a\mid \\&-2u\,(n_a+n_{{a}+1})-2\displaystyle{\sum_{i=u+2}^r(n_{b_i}+n_{b_i+1})}.
\end{array}$$
Since by Equation \ref{bu}, $a\geq 2u$, we then get
$$\begin{array}{ll}
\mid U_{b_1, \cdots,b_{u},a,b_{u+2},\cdots,  b_r}\mid -\mid U_{b_1, \cdots, b_r}\mid &=\mid U_a\mid -\mid U_{b_{u+1}}\mid -2u\,(n_a+n_{a+1}-n_{b_{u+1}}-n_{b_{u+1}+1})\\&=(a-2u)(n_{a}-n_{a+2})\\&=(\frac{a-2u}{a})(\mid U_a\mid -\mid U_{b_{u+1}}\mid )\geq 0.
\end{array}$$
\medskip
{\bf Case 2.2.} If $b_u+1=a<b_{u+1}-1$. 

Then by maximality of $\mid U_a\mid $, we have $$0\leq \mid U_a\mid -\mid U_{b_{u}}\mid =\mid U_a\mid -\mid U_{a-1}\mid =(a-1)\,(n_{a+1}-n_{a-1}).$$

By Lemma \ref{chain length lem}, 
$$\begin{array}{ll}
\mid U_{b_1, \cdots,b_{u-1},a,b_{u+1},\cdots,  b_r}\mid -\mid U_{b_1, \cdots, b_r}\mid &=\mid U_a\mid -\mid U_{b_{u}}\mid -2(u-1)(n_{a+1}+n_{a-1}-n_{b_{u}+1}-n_{b_{u+1}})\\&=(a-2u+1)(n_{a+1}-n_{a-1})\\&=(\frac{a-2u+1}{a-1})(\mid U_a\mid -\mid U_{b_{u}}\mid ).
\end{array}$$

By Equation \ref{bu}, we have $a\geq 2u-1$, and therefore $(\frac{a-2u+1}{a-1})(\mid U_a\mid -\mid U_{b_{u}}\mid )\geq 0,$ as desired.
\bigskip 

{\bf Case 2.3.} If $b_u+1 < a < b_{u+1}-1$. 

Let $b=b_u$, and $\Delta=\mid U_{b_1, \cdots,b_{u-1},a,b_{u+1}, b_r}\mid -\mid U_{b_1, \cdots, b_r}\mid $. By Lemma \ref{chain length lem}, we have 
$$\begin{array}{ll}
\Delta&=\mid U_a\mid -\mid U_{b}\mid -2(u-1)(n_{a+1}+n_{a}-n_{b+1}-n_{b})\\
&=(a-2u+1)n_{a+1}+(a-2u)n_{a}-2\displaystyle{\sum_{p=b+2}^{a-1}n_p}-(b-2u+3)n_{b+1}-(b-2u+2)n_{b}.
\end{array}$$
We will prove that $\Delta \geq 0.$ 

For $c\in\{b, \cdots, a-1\}$, define $\delta_{c}=\mid U_a\mid -\mid U_c\mid $. By the maximality of $\mid U_a\mid $, we have $\delta_c \geq 0$ for all $c$. We also have $\delta_{a-1}=(a-1)n_{a+1}-(a-1) n_{a-1},$ and if $b\leq c<a-1$, then 
$$\delta_c= (a-1)n_{a+1}+(a-2)n_a-2(\displaystyle{\sum_{p=c+2}^{a-1}n_p}) -(c+1)n_{c+1}-c\,n_{c}.$$

We define an $(a-b+2) \times (a-b)$ matrix $M$ such that for all $1\leq j\leq a-b$, $$\delta_{a-j}=\displaystyle{\sum_{i=1}^{a-b+2}M_{ij}(n_{a-i+2})}.$$ So we have 
$$M=\left( \begin{array}{cccccc}
a-1    & a-1    & a-1 & \cdots& a-1   \\
0        & a-2    & a-2& \cdots & a-2   \\
-(a-1)& -(a-1)& -2   & \cdots  &-2     \\
     0   & -(a-2)& -(a-2) & \cdots & -2   \\
 \vdots &   \vdots & \vdots & \vdots & \vdots   \\
     0   &      \cdots &    0 & -(b+1) & -(b+1)    \\
     0   &      \cdots  &     0 & 0  & -b     
\end{array} \right). $$

Let $$R=\left( \begin{array}{c} r_{1}  \\ \vdots  \\  r_{a-b} \end{array} \right) \mbox{ and } D=  \left( \begin{array}{c}  a-2u+1\\  a-2u \\ -2 \\ \vdots \\   -2  \\  -(b-2u+3) \\ -(b-2u+2) \end{array} \right).$$ Note that $D$ is defined such that $$\Delta =\sum_{i=1}^{a-b+2} D_i  \cdot (n_{a-i+2}).$$

We will show that the linear system $M\cdot R=D$ of linear equations, has a unique non-negative solution. This implies $\Delta=\displaystyle{\sum_{i=1}^{a-b}r_i\, \,\delta_{a-i}}$, with $r_i \geq 0$ for all $i$, which proves the desired inequality $\Delta \geq 0$.

Let $M=M(1)$ and $D=D(1)$, and for $k=1, \cdots,[\frac{a-b}{2}]+1 $, let $M(k+1)$ (respectively $D(k+1)$) denote the matrix obtained by adding the $(2k-1)$-st row of $M(k)$ (respectively $D(k)$) to its $(2k+1)$-st row and adding the $2k$-th row of 
$M(k)$ (respectively $D(k)$) to its $(2k+2)$-nd row. Then for all $k$, the linear system $M(k) \cdot R=D(k)$ of linear equation is equivalent to the linear system $M \cdot R=D$. For $\ell=[\frac{a-b}{2}]+1$, we have  

$$M(\ell)=\left( \begin{array}{cccccc}
a-1    & a-1    & a-1 & a-1 & \cdots & a-1 \\
0        & a-2    & a-2 & a-2 & \cdots & a-2 \\
0       & 0      & a-3    &a-3     & \cdots & a-3 \\
 \vdots &   \vdots & \vdots & \vdots & \vdots & \vdots \\
      0   & \cdots & 0 & b+1 & \cdots & b+1 \\

     0   &      \cdots &    0 & 0 & 0 & b \\
     0   &      \cdots  &     \cdots  & 0     & 0 & 0 \\
  0   &      \cdots  &    \cdots  &  \cdots  & 0     & 0
\end{array} \right), \mbox{ and }
D(\ell)=\left( \begin{array}{c} a-2u+1 \\  a-2u \\ a-2u-1 \\  \vdots \\  b-2u+3 \\   b-2u+2  \\  0 \\ 0 \end{array} \right).$$

Therefore, to prove the claim it is enough to prove that the the following linear system of $a-b$ equations in $a-b$ variables has a non-negative solution.

$$\left( \begin{array}{cccccc}
1   & 1    & 1 & \cdots & 1 \\ 
0        & 1    & 1  & \cdots & 1 \\ 
0       & 0      & 1        & \cdots & 1 \\ 
 \vdots &   \vdots  & \vdots & \vdots & \vdots \\ 
     0   &      \cdots &    0 & 0 & 1 \\

\end{array} \right)
 \cdot R = \left( \begin{array}{c}1-\frac{2(u-1)}{a-1}  \\ 1-\frac{2(u-1)}{a-2}  \\\vdots \\   1-\frac{2(u-1)}{b+1} \\ 1-\frac{2(u-1)}{b}  \end{array} \right).$$

This system has the following unique solution:
$$
r_{a-b}=1-\displaystyle{\frac{2(u-1)}{b}},\mbox{ and for } 1\leq i < a-b,\, \,
r_{i}=\displaystyle{\frac{2(u-1)}{(a-i)(a-i-1)}}.
$$

Note that by Equation \ref{bu}, $r_{a-b}>0$. Clearly, $r_i \geq 0$ for $1\leq i < a-b$, as well. This completes the proof of the proposition.
\end{proof}

\begin{remark}
As the proof of Proposition \ref{U chain} shows, the proposition holds if the cardinality of $U_a$ is greater than or equal to the cardinality of $U_c$ for all $b_1 \leq c \leq b_r+1$.
\end{remark}

\begin{example}
Let $P=(6^2,5,4,3,2^2,1^2)$. We will examine Proposition \ref{U chain} for the $2$-$U$-chain $U_{1,3}$ in $\mathcal{D}_P$. First note that $U_5$ with cardinality 17 is the only maximum simple $U$-chain in $\mathcal{D}_P$. Replacing $U_3$ with $U_5$ in $U_{1,3}$, we can obtain a larger 2-$U$-chain, as we have $27=\mid U_{1,5}\mid >\mid U_{1,3}\mid =25$. Also note that the proposition is an existence result and is not necessarily true for all $u$. In fact in this example we have $24=\mid U_{3,5}\mid <\mid U_{1,3}\mid =25$.  

%
%
%
%
%
%
%

\end{example}

\section{Uniqueness of Oblak Partitions}\label{O}

In this section we discuss a recursive process, which was originally defined by P. Oblak and later generalized in \cite{BIK}. A generalized Oblak process, or a $U$-process, is a recursive process defined by finding a maximum simple $U$-chain in the poset corresponding to a partition, then obtain a new partition by removing the elements of this simple chain from the poset, and then repeat the same process.

Let $P=(p_s^{n_s}, \cdots, p_1^{n_1})$ be a partition of $n$. Suppose that $a$ is a positive integer and consider the simple $U$-chain $C=U_a=U_{\{a,a+1\}}$ in $\mathcal{D}_P$. Let $P'$ be the partition corresponding to the vertices of $\mathcal{D}_{P} \setminus C$. Namely $P'=(q_s^{m_s}, \cdots, q_1^{m_1})$,  such that  $$q_i=\left\{\begin{array}{lll}p_i&&\mbox{if }p_i<a\\ p_i-2&&\mbox{if }p_i>a+1\end{array}\right.; \mbox{ and }m_i=\left\{\begin{array}{lll}n_i&&\mbox{if }p_i\not \in \{a,a+1\}\\ 0&&\mbox{if }p_i\in\{a,a+1\}\end{array}\right..$$

Then there is a natural ``relabeling" map of sets $\iota:\mathcal{D}_{P'}\to \mathcal{D}_{P}$ defined as follows.
\begin{equation}\label{Oblak pi}
\iota(\, (u, p, k)\, )=\left\{ \begin{array}{lll} 
(u,p,k) && \mbox{ if } p < a, \\
 (u+1, p+2, k )&&\mbox{ if } p\geq a. \\
\end{array}
\right.
\end{equation}

%

\begin{definition}\label{Oblak defn}

A {\it $U$-process} for $P$ is a succession $\frak{C}=(C_1, \cdots, C_m)$ of subsets of $\mathcal{D}_P$ defined recursively as follows.
\begin{itemize}
\item $P_1=P$ and $\iota_1$ is the identity map.  
\item $C^\dagger_i$ is a maximum simple $U$-chain in $\mathcal{D}_{P_i}$ and $C_i=\iota_1\cdots\iota_i(C^\dagger_i)$.
\item $P_{i+1}$ is the partition obtained from the diagram of $\mathcal{D}_{P_i}$ after removing $C_i^\dagger$, and $\iota_{i+1}:\mathcal{D}_{P_{i+1}} \to \mathcal{D}_{P_{i}}$, is defined as in Equation \ref{Oblak pi} above.
\end{itemize}

A $U$-process $\frak{C}=(C_1, \cdots, C_r)$ is called {\it full} if $C_1 \cup \cdots \cup C_r=\mathcal{D}_P.$ To each full $U$-process $\frak{C}$, we assign a partition $Q_{\frak{C}}(P)=(\mid C_1\mid , \cdots, \mid C_r\mid )$ of $n$.

\end{definition}

\begin{remark}\label{Oblak min remark}
Assume that $\frak{C}=(C_1, \cdots, C_r)$ is a full $U$-process for $P$. By definition,  $C_r$ is the pullback of the vertices of a maximum simple $U$-chain of $\mathcal{D}_{P_{r}}$ into $\mathcal{D}_P$. Since $\frak{C}$ is full, $\mathcal{D}_{P_{r}}$ must be a simple $U$-chain. Thus $P_{r}$ is an almost rectangular partition. 
\end{remark}

A given partition $P$ may admit several full $U$-processes, as the following example shows. In \cite{Oblak}, P. Oblak picks a particular $U$-process, choosing the maximum chain above all others in the diagram of $\mathcal{D}_P$ at each step, and conjectures that the corresponding partition is the same as $Q(P)$ (see \cite{BKO}). 

\begin{example}\label{new bridge}
(See Figure 3.) Let $P=(5,4,3^2,2,1)$. Then both $C_1=U_3$ and $D_1=U_2$ are maximum simple $U$-chains in $\mathcal{D}_P$ (both of length 12). So one can begin a $U$-process with either one of those. We have $\mathcal{D}_P \setminus C_1=\{(2,5,1),(3,5,1),(4,5,1),(1,2,1),(2,2,1),(1,1,1)\} $, and therefore the corresponding partition $P_2=(3,2,1)$. Again at this point, we have two choices for a maximum simple $U$-chain, namely $U_2$ or $U_1$ which correspond to  
$$\begin{array}{l}
C_2=\{(2,5,1),(1,2,1),(3,5,1),(2,2,1),(4,5,1)\},\mbox{  and} \\
C'_2=\{(2,5,1),(1,2,1),(1,1,1),(2,2,1),(4,5,1)\},
\end{array}$$ in $\mathcal{D}_P$, respectively. These choices give rise to two different full $U$-processes 
$\frak{C}=(C_1, C_2, C_3)$, and 
$\frak{C}'=(C_1,C'_2,C'_3)$, where $C_3=\{(1,1,1)\}$ and $C'_3=\{(3,5,1)\}$.

On the other hand, starting a $U$-process with $D_1$, we can get two other full $U$-processes 
$\frak{D}=(D_1,D_2,D_3)$ and $\frak{D}'=(D_1,D'_2,D'_3)$, where 
$$\begin{array}{l}
D_2=\{(2,5,1),(2,4,1),(3,5,1),(3,4,1),(4,5,1)\},\\
D_3=\{(1,1,1)\},\\
D'_2=\{(2,5,1),(2,4,1),(1,1,1),(3,4,1),(4,5,1)\},\\
D'_3=\{(3,5,1)\}.
\end{array}$$
\end{example}

\begin{figure}\label{New Oblak figure}
\begin{center}
\vspace{-.3in}
\includegraphics[scale=.5]{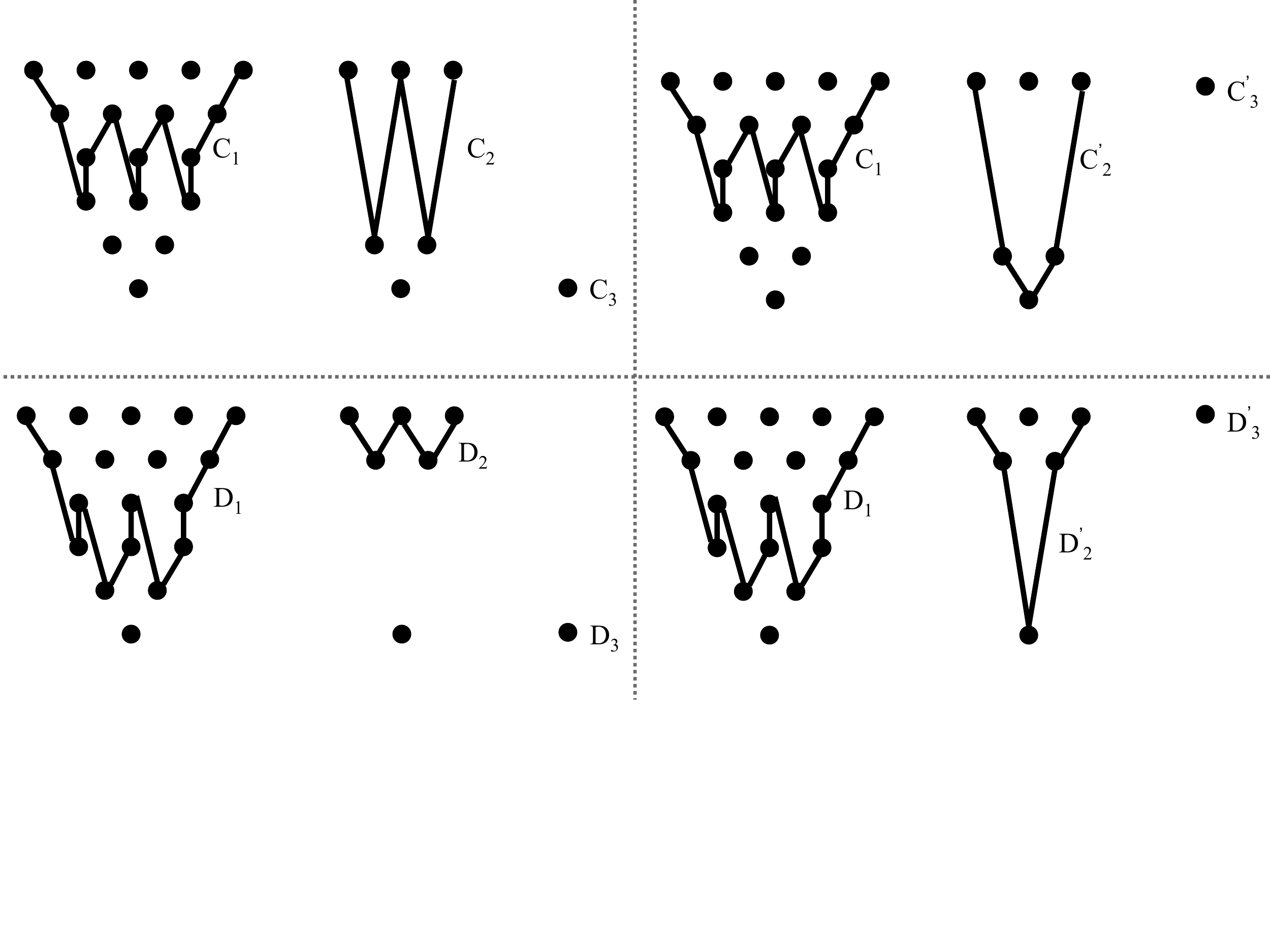}
\end{center}
\vspace{-.8in}
\caption{Different full $U$-processes for $P=(5,4,3^2,2,1)$}

\end{figure}


Although the partition $P$ above admits four different full $U$-processes, the corresponding partitions are all the same, namely $(12,5,1)$.  In Theorem \ref{unique Oblak}, we prove that every full $U$-process give rise to the same partition of $n$, which is in fact equal to the partition $\lambda_U(P)$ introduced in Definition \ref{lambda U}.

It is also worth noting that in general, for a $U$-process $\frak{C}=(C_1, \cdots, C_r)$, the $C_k$'s may not be chains in $\mathcal{D}_P$, since $\mathcal{D}_{P_k}$ is not necessarily a sub poset of $\mathcal{D}_P$. This is easy to observe in Example \ref{new bridge}. For example, both $(2,5,1)$ and $(1,2,1)$ belong to $C_2$ but they are not comparable in $\mathcal{D}_P$. On the other hand, $C_1 \cup C_2$ is a union of two chains in $\mathcal{D}_P$. In fact,  $C_1 \cup C_2=U_{2,4}$ is a 2-$U$-chain. The following proposition shows that this is the case at any given step of a $U$-process. This is also stated without a proof in \cite[Proposition 3.18]{BIK}.

%
%
%
%
%

\begin{proposition} \label{BIK shelling}
Let $P$ be a partition and $\frak{C}=(C_1, \cdots, C_r)$ a $U$-process for $P$. There exists an $r$-$U$-chain $U_{b_1, \cdots,b_r}$ in $\mathcal{D}_P$ such that $C_1 \cup \cdots \cup C_r=U_{b_1, \cdots,b_r}$, as sets.
\end{proposition}
\begin{proof}
Suppose that $P=(p_s^{n_s}, \cdots, p_1^{n_1})$ such that $p_s>\cdots>p_1$ and $n_i>0$, for $1\leq i \leq s$. We give a proof by induction on $r$. 

By the definition of a $U$-process, there is nothing to prove for $r=1$. Assume that $r>1$, $\frak{C}=(C_1, \cdots, C_r)$ is a $U$-process for $P$, and  $C_1 \cup \cdots \cup C_{r-1}=U_{b_1, \cdots,b_{r-1}}$, for some $b_1, \cdots,b_{r-1}$ such that $b_1<b_1+1 <\cdots<b_{r-1}<b_{r-1}+1$. By the definition of a $U$-process, $C_r=\iota_1\cdots\iota_r(C_r^\dagger)$, where $C_r^\dagger$ is a maximum simple $U$-chain in the $\mathcal{D}_{P_{r}}$. Since $C_1 \cup \cdots \cup C_{r-1}=U_{b_1, \cdots,b_{r-1}}$, we can write $P_{r}=(q_s^{m_s}, \cdots, q_1^{m_1})$ such that
$$q_i=p_i-2r_i, \mbox{ where }r_i=\mid \{b_j\,\mid \, 1\leq j \leq r-1 \mbox{ and } b_j<p_i\}\mid ); \mbox{ and }$$
$$m_i=\left\{\begin{array}{lll}
0&&\mbox{ if } p_i\in\{b_1, b_1+1, \cdots, b_{r-1},b_{r-1}+1\}\\
n_i&&\mbox{ otherwise.}
\end{array}\right.$$
Suppose that $C^\dagger_r$ is the simple $U$-chain $U_a$ in $\mathcal{D}_{P_{r}}$. So there exists an integer $u$ such that $a=q_{u}$ with $m_u\neq 0$. Thus $C_1 \cup \cdots \cup C_r$ is equal, as a set, to the $r$-$U$-chain $U_\frak{B}$ where $\frak{B}=\{b_1,b_1+1, \cdots, b_{r-1},b_{r-1}+1,p_u,p_u+1\}$, as desired.

\end{proof}

\begin{theorem} \label{unique Oblak}

Let $P$ be a partition and $\frak{C}=(C_1, \cdots, C_r)$ a $U$-process for $P$. Then $$\mid \cup_{i=1}^{r} C_i\mid =\max \{\mid U_\frak{B}\mid  \mbox{ such that } U_\frak{B} \mbox{ is an }r\mbox{-}U\mbox{-chain in }  \mathcal{D}_P\}.$$

In particular, for any full $U$-process $\frak{C}$ of $P$,  we have $Q_{\frak{C}}(P)=\lambda_U(P)$.
\end{theorem}

\begin{proof}

By Proposition \ref{BIK shelling}, there exist $a_1, \cdots, a_r$ such that $a_1<a_1+1<\cdots<a_r<a_r+1$ and $C_1 \cup \cdots \cup C_r= U_{a_1, \cdots \, a_{r}}$. Therefore, \begin{equation}\label{ey baba}\mid \cup_{i=1}^{r} C_i\mid  \leq \max \{\mid U_\frak{B}\mid  \mbox{ such that } U_\frak{B} \mbox{ is an }r\mbox{-}U\mbox{-chain in }  \mathcal{D}_P\}.\end{equation}

We prove the converse inequality by induction on $r$. 

For $r=1$, the claim is clear by definition. Now suppose that $r>1$ and that for any $m<r$ and any $U$-process $\frak{C'}=(C'_1, \cdots, C'_m)$ of a partition $P'$, the desired equality holds. 

Assume that $U_{b_1, \cdots, b_r}$ is an arbitrary $r$-$U$-chain in $\mathcal{D}_P$. Since $\frak{C}=(C_1, \cdots, C_r)$ is a $U$-process for $P$, $C_1$ is a maximum simple $U$-chain in $\mathcal{D}_P$. Suppose that $C_1=U_a$. Then, by Proposition~\ref{U chain}, there exists a $u$ such that $1 \leq u \leq r$ and 
\begin{equation}\label{yek}\mid U_{b_1, \cdots, b_r}\mid  \leq \mid U_{b_1, \cdots, b_{u-1},a,b_{u+1}, \cdots, b_r}\mid .\end{equation}

Let $P'$ denote the partition corresponding to the vertices in $\mathcal{D}_P\setminus U_a$ and $\iota:\mathcal{D}_{P'}\to \mathcal{D}_{P}$ be the relabeling map given by Equation \ref{Oblak pi}. By definition of a $U$-process, $\cup_{i=2}^rC_i \subseteq \mathcal{D}_P\setminus U_a$, and therefore $\frak{C}'=(\iota^{-1}(C_2), \cdots, \iota^{-1}(C_r))$ is a $U$-process for $P'$. Also note that $\iota^{-1}(U_{b_1, \cdots, b_{u-1},a,b_{u+1}, \cdots, b_r})$ is the $(r-1)$-$U$-chain $U'=U_{b_1, \cdots, b_{u-1},b_{u+1}-2, \cdots, b_r-2}$ in $\mathcal{D}_{P'}$. Thus, by the inductive hypothesis, 
\begin{equation}\label{dow}
\mid \cup_{i=2}^r C_i\mid =\mid \cup_{i=2}^r \iota^{-1}(C_i)\mid  \geq \mid U'\mid .
\end{equation}

On the other hand, by definition of a $U$-chain (Definition \ref{general U chain}), $U_{b_1, \cdots, b_{u-1},a,b_{u+1}, \cdots, b_r}$ is the union of the two disjoint sets $U_a$ and $\iota(U')$. Therefore  
\begin{equation}\label{se}
\mid U_{b_1, \cdots, b_{u-1},a,b_{u+1}, \cdots, b_r}\mid =\mid U_a\mid +\mid U'\mid .
\end{equation}

Thus 
$$\begin{array}{llll}
\mid \cup_{i=1}^r C_i\mid &=\mid U_a\mid +\mid \cup_{i=2}^r C_i\mid &&\\
&\geq \mid U_a\mid + \mid U'\mid &&\mbox{(By Equation \ref{dow})}\\
&=\mid U_{b_1, \cdots, b_{u-1},a,b_{u+1}, \cdots, b_r}\mid &&\mbox{(By Equation \ref{se})}\\
&\geq \mid U_{b_1, \cdots, b_r}\mid .&&\mbox{(By Equation \ref{yek})}\\
\end{array}$$
This completes the proof of the theorem.
 \end{proof}

%
%
%

In view of Theorem \ref{unique Oblak}, Oblak's conjecture can be restated as $Q(P)=\lambda_U(P)$. We conclude this paper by showing that $\lambda_U(P)$ shares another property of the partition $Q(P)$, namely parts of the partition $\lambda_U(P)$ differ pairwise by at least 2. See \cite[Theorem 6]{KO} for the corresponding result for $Q(P)$ when char $\k=0$ and \cite[Theorem 1]{BI} when char $\k>n$.

\begin{lemma}\label{max U chain}
Let $n>1$ and $P=(\ldots, p^{n_p}, \ldots)$ be a partition of $n$. If $U_a$ is a maximum $U$-chain in $\mathcal{D}_P$ then $ \mid U_a \mid \geq 2,$ and $n_a+n_{a+1}>0.$
\end{lemma}
\begin{proof}
Suppose that $b=\max\{b\,\mid \, n_b>0\}$. If $b=1$, then $n_b$ must be at least 2, and therefore $bn_b\geq 2$. If $b\geq 2$, then $bn_b\geq 2$. Thus $\mid U_a\mid\geq \mid U_b\mid\geq 2$.

We prove the second inequality by contradiction. Assume that $n_a=n_{a+1}=0$, then $2 \leq \mid U_a\mid=2\displaystyle{\sum_{p>a+1}n_p}$. Let $c=\min\{p\,\mid \, p>a+1 \mbox{ and } n_p>0\}$.  Since $c>a+1\geq 2$, we get $\mid U_a \mid < \mid U_{c}\mid $. This contradicts the maximality of $\mid U_a\mid$. So $n_a+n_{a+1}$ must be positive.

\end{proof}

\begin{proposition}
Let $n>1$ and $P=(\cdots, p^{n_p}, \cdots)$ be a partition of $n$. Then the parts of $\lambda_U(P)$ differ pairwise by at least 2.
\end{proposition}
\begin{proof}
By Theorem \ref{unique Oblak} and the inductive definition of a $U$-process, it is enough to prove the following claim.

\noindent{\bf Claim.} Let $U_a$ be a maximum simple $U$-chain in $\mathcal{D}_P$ and let $P'$ be the partition corresponding to the vertices in $\mathcal{D}_P\setminus U_a$. If $U'$ is a simple $U$-chain in $\mathcal{D}_{P'}$, then $\mid U_a \mid \geq \mid U' \mid+2.$ 

By Lemma \ref{max U chain}, there is nothing to prove if $U'$ is empty. So assume that $U'=U_{b}$ in $\mathcal{D}_{P'}$ is not empty. We have
$$\begin{array}{lll}
U'&=\{(u,p,k)\in \mathcal{D}_{P'}\,\mid \, p\in\{b, b+1\} \mbox{ and } 1\leq u \leq p\}\\
&\cup \{(u,p,k)\in \mathcal{D}_{P'}\,\mid \, p>b+1 \mbox{ and } u\in\{1,p\} \}.
\end{array}$$

Recall that the relabeling map of Equation \ref{Oblak pi} is an injective map from $\mathcal{D}_{P'}$ to $\mathcal{D}_{P}$.

 \medskip
 
{\bf Case 1.} If $a>b+1$, then 
$$\begin{array}{ll}
\mid U' \mid=\mid \iota(U')\mid&=\displaystyle{bn_b+(b+1)n_{b+1}+2\sum_{b+1<p<a}n_p+2\sum_{p\geq a+2}n_p}\\
&= \mid U_b \mid-2(n_a+n_{a+1}).
\end{array}$$ 
Here $U_b$ is the simple $U$-chain in $\mathcal{D}_P$. By Lemma \ref{max U chain}, and maximality of $\mid U_a \mid$, we get $\mid U' \mid \leq \mid U_b \mid-2\leq \mid U_a \mid -2.$
This completes the proof of the claim in this case.
\medskip

{\bf Case 2.} If $a=b+1$, then 
$$\begin{array}{ll}
\mid U' \mid=\mid \iota(U')\mid&=bn_{b}+(b+1)n_{b+3}+2\displaystyle{\sum_{p>b+3}n_p}\\
&=(a-1)n_{a-1}+a\, n_{a+2}+2\displaystyle{\sum_{p>a+2}n_p}.
\end{array}$$
{\bf Case 2.1.} If $n_{a+1}=n_{a+2}=0$. Then by Lemma \ref{max U chain} $n_a>0$, and since $a=b+1\geq 2$, we have $an_a\geq 2$. So we get the desired inequality $$\mid U' \mid=\mid U_{a-1} \mid -a\, n_a\leq \mid U_a \mid -2.$$

{\bf Case 2.2.} If $n_{a+1}+n_{a+2}>0$. Then $$\mid U' \mid=\mid U_{a-1} \mid-\mid U_a \mid +\mid U_{a+1}\mid -2(n_{a+1}+n_{a+2})\leq \mid U_a \mid-2.$$
This completes the proof of Case 2.
\medskip

{\bf Case 3.} If $a\leq b$, then 
$$\begin{array}{ll}
\mid U' \mid=\mid \iota(U')\mid&=\displaystyle{b\, n_{b+2}+(b+1)n_{b+3}+2\sum_{p>b+3}n_p}\\
&= \mid U_{b+2} \mid-2(n_{b+2}+n_{b+3}).
\end{array}$$ 
{\bf Case 3.1.} If $n_{b+2}=n_{b+3}=0$. Then $0<\mid U' \mid =2\displaystyle{\sum_{p>b+3}n_p}.$ Let $c=\min\{p\,\mid \, p>b+3 \mbox{ and } n_p>0\}.$ Then $\mid U' \mid=\mid U_{c-1} \mid -(c-2)n_c$. Since $c>b+3\geq 4$ and $n_c>0$, we get $$\mid U' \mid=\mid U_{c-1} \mid -(c-2)n_c \leq \mid U_{c-1} \mid-2 \leq \mid U_a \mid -2,$$ as desired.

{\bf Case 3.2.} If $n_{b+2}+n_{b+3}>0$, then the desired inequality is clear by maximality of $\mid U_a \mid$.
This completes the proof of the proposition.

\end{proof}

In \cite{minpart} we further study the poset $\mathcal{D}_P$ and the partition $\lambda_U(P)$ and give an explicit formula for its smallest part $\mu(P)$. By enumerating the disjoint maximum antichains in $\mathcal{D}_P$ and use of results from $\cite{Oblak}$ and \cite{IK}, we prove that the smallest part of $Q(P)$ is $\mu(P)$ as well. This, combined with Oblak's formula for the index of $Q(P)$ (\cite[Theorem 6]{Oblak} for char $\k=0$, and \cite[Corollary 3.10]{IK} for any infinite field $\k$), gives an explicit formula for $Q(P)$, when it has at most 3 parts ({\it i.e.} when $P$ can be written as a union of 3 almost rectangular sub partitions).

\bibliography{ref}

\begin{thebibliography}{10}

\bibitem{Basili03}
R.~Basili.
\newblock On the irreducibility of commuting varieties of nilpotent matrices.
\newblock {\em J. Algebra}, 268(1):58--80, 2003.

\bibitem{BI}
R.~Basili and A.~Iarrobino.
\newblock Pairs of commuting nilpotent matrices, and {H}ilbert function.
\newblock {\em J. Algebra}, 320(3):1235--1254, 2008.

\bibitem{BIK}
R.~Basili, A.~Iarrobino, and L.~Khatami.
\newblock Commuting nilpotent matrices, and {A}rtinian algebras.
\newblock {\em J. Commut. Algebra}, 2(3):295--325, 2010.

\bibitem{BKO}
R.~Basili, T.~Ko{\v{s}}ir, and P.~Oblak.
\newblock Some ideas from {L}jublijana.
\newblock {\em Preprint}, 2008.

\bibitem{Britz-Fomin}
T.~Britz and S.~Fomin.
\newblock Finite posets and {F}errers shapes.
\newblock {\em Advances Math.}, 158:86--127, 2001.

\bibitem{Greene}
C.~Greene.
\newblock Some partitions associated with a partially ordered set.
\newblock {\em J. Comb. Theory, Ser. A}, 20:69--79, 1976.

\bibitem{HW}
T.~Harima and J.~Watanabe.
\newblock The commutator algebra of a nilpotent matrix and an application to
  the theory of commutative {A}rtinian algebras.
\newblock {\em J. Algebra}, 319(6):2545--2570, 2008.

\bibitem{IK}
A.~Iarrobino and L.~Khatami.
\newblock Bound on the {J}ordan type of a generic nilpotent matrix commuting
  with a given matrix.
\newblock {\em arXiv:1204.4635}, 2012.

\bibitem{minpart}
L.~Khatami.
\newblock The smallest part of the generic partition of the nilpotent
  commutator of a nilpotent matrix.
\newblock {\em Preprint}, 2012.

\bibitem{KO}
T.~Ko{\v{s}}ir and P.~Oblak.
\newblock On pairs of commuting nilpotent matrices.
\newblock {\em Transform. Groups}, 14(1):175--182, 2009.

\bibitem{Oblak}
P.~Oblak.
\newblock The upper bound for the index of nilpotency for a matrix commuting
  with a given nilpotent matrix.
\newblock {\em Linear Multilinear Algebra}, 56(6):701--711, 2008.

\bibitem{Oblak2}
P.~Oblak.
\newblock On the nilpotent commutator of a nilpotent matrix.
\newblock {\em Linear Multilinear Algebra}, 60(5):599--612, 2012.

\bibitem{Pa}
D.~I. Panyushev.
\newblock Two results on centralisers of nilpotent elements.
\newblock {\em J. Pure Appl. Algebra}, 212(4):774--779, 2008.

\bibitem{Saks}
M.~Saks.
\newblock Dilworth numbers, incidence maps and product partial orders.
\newblock {\em SIAM J. Alg. Discr. Meth.}, 1:211--215, 1980.

\end{thebibliography}
\bibliographystyle{abbrv}
\nocite{Oblak2}
\bigskip

{\sc Department of Mathematics, Union College, Schenectady, NY 12308} \par
{\it E-mail Address}: khatamil@union.edu
\end{document}